\documentclass{amsart}

\usepackage{amsmath}
\usepackage{amssymb,amsfonts,bm}
\usepackage{color,graphicx,caption,subfig}
\usepackage[draft]{hyperref}

\newtheorem{thm}{Theorem}[section]
\newtheorem{prop}[thm]{Proposition}
\newtheorem{lem}[thm]{Lemma}

\theoremstyle{remark}
\newtheorem{remark}[thm]{Remark}

\usepackage{xspace}

\def \ud {\,\textnormal{d}}

\newcommand{\argsup}[1]{\underset{#1}{\textnormal{arg sup }}}

\definecolor{lgray}{gray}{0.9}
\definecolor{llgray}{gray}{0.95}
\definecolor{lllgray}{gray}{0.975}

\newcommand{\mybox}[2]{
\begin{center}
\begin{minipage}[t]{\dimexpr\textwidth}
\begin{center}
\fcolorbox{black}{lgray}{%
\begin{minipage}[t]{\dimexpr\textwidth-4\fboxsep-4\fboxrule}
#1
\end{minipage}
}
\fcolorbox{black}{llgray}{%
\begin{minipage}[t]{\dimexpr\textwidth-4\fboxsep-4\fboxrule}
#2
\end{minipage}
}
\end{center}
\end{minipage}
\end{center}
\vspace{11pt}
}

\newcommand{\rA}{{\rm A}}
\newcommand{\rC}{{\rm C}}

\newcommand{\rM}{{\rm M}}
\newcommand{\rR}{{\rm R}}
\newcommand{\rU}{{\rm U}}
\newcommand{\rV}{{\rm V}}
\newcommand{\rL}{{\rm L}}
\newcommand{\rI}{{\rm I}}

\newcommand{\rh}{{\rm h}}

\newcommand{\rr}{{\rm r}}
\newcommand{\rv}{{\rm v}}
\newcommand{\ru}{{\rm u}}

\newcommand{\re}{{\rm e}}

\newcommand{\real}{\mathbb R}

\newcommand{\err}{{\tt err}}
\newcommand{\tol}{{\tt tol}}

\newcommand{\Oo}{{\Omega_{\tt x}}}
\newcommand{\Ot}{{\Omega_{\tt y}}}
\newcommand{\Oot}{\Omega_{\tt x}^{\tt train}}
\newcommand{\Ott}{\Omega_{\tt y}^{\tt train}}
\newcommand{\Oop}[1]{{\Omega_{\tt x}^{#1}}}
\newcommand{\Otp}[1]{{\Omega_{\tt y}^{#1}}}

\newcommand{\Inter}{I}
\newcommand{\gInter}{J}
\newcommand{\MInter}{{\rm I}}
\newcommand{\proj}{{ P}}

\newcommand{\argmin}[1]{\operatorname*{arg\,min}_{#1}}
\newcommand{\argmax}[1]{\operatorname*{arg\,max}_{#1}}

\newcommand{\myspan}[1]{\mbox{span}\{#1\}}
\newcommand{\cN}{M}
\newcommand{\VN}{\mathbb V_\cN}
\newcommand{\VQ}{\mathbb V_Q}
\newcommand{\WQ}{\mathbb W_Q}
\newcommand{\Vq}{\mathbb V_q}

\newcommand{\Vpod}{\mathbb V_Q}

\newcommand{\myinclude}[2]{{\includegraphics{pics/ellipses_a_N10_Ex3}}}

\title[]{Comparison of some reduced representation approximations}
\author{Mario Bebendorf}
\address{Mario Bebendorf - Institute for Numerical Simulation, University of Bonn, Wegelerstra{\ss}e 6, 53115 Bonn, Germany}

\author{Yvon Maday} 
\address{Yvon Maday - UPMC Univ. Paris 06, UMR 7598 LJLL, Paris, F-75005 France; \\  Institut Universitaire de France and \\ Division of Applied Mathematics, Brown
  University, Providence, RI, USA} 

\author{Benjamin Stamm}
\address{Benjamin Stamm - UPMC Univ. Paris 06, UMR 7598 LJLL, Paris, F-75005 France; \\ CNRS, UMR 7598 LJLL, Paris, F-75005 France}

\begin{document}
\maketitle

\begin{abstract}
In the field of numerical approximation, specialists considering highly complex problems have recently proposed various ways to simplify their underlying problems. 
In this field, depending on the problem they were tackling and the community that are at work, different approaches have been developed with some success and have even gained some maturity, the applications can now be applied to information analysis or for numerical simulation of PDE's. 
At this point, a crossed analysis and effort for understanding the similarities and the differences between these approaches that found their starting points in different backgrounds is of interest. It is the purpose of this paper to contribute to this effort by comparing some constructive reduced representations of complex functions. We present here in full details the Adaptive Cross Approximation (ACA) and the Empirical Interpolation Method (EIM) together with other approaches that enter in the same category.
\end{abstract}

\section{Introduction}

This paper deals with the economical representation of {\it dedicated} sets of data, that are currently --- and more and more importantly --- available stemming out of various experiences or given by formal expressions. The amount of information that can be derived out of a given massive set of data is far much smaller than the size of the data itself, therefore, parallel to the increasing size of data acquisition and storage available on computer architectures, an effort for post processing and economically represent, analyze and derive pertinent information out of the data has been done during the last century. The main idea starts from the translation of the fact that the data are {\it dedicated} to some phenomenon and thus, there exists a certain amount  of {\it coherence} in these data which can be separated into two classes: deterministic or  statistical. Among them have been proposed: regularity, sparsity, small $n$-width etc.\ that can be either assumed, verified or proven.

The data themselves can be known in different ways, either (i) completely explicitly, like for instance (i-1) from an analytic representation or at least access to the values at every point, (i-2) or only given on a large set of points, (i-3) or also given through various global measures like moments, or (ii) given implicitly through a model like a partial differential equation (PDE). The range of  applications is huge, examples can be found in statistics, image and information process, learning process, experiments in mechanics, meteorology, earth sciences, medicine, biology, etc.\ and the challenge is in computationally processing such a large amount of high-dimensional
data so as to obtain low-dimensional descriptions and capture much of the phenomena of interest.

\bigskip

We consider the following problem formulation: Let us assume that we are given a (presumably large) set $\mathcal F$ of functions $\varphi\in \mathcal F$ defined over $\Oo\subset \mathbb R^{d_x}$ (with $d_x\ge1$). Our aim is to find some functions $h_1, h_2,\dots, h_Q:\Oo\to \mathbb R$ such that every $\varphi\in \mathcal F$ can be well approximated as follows
\[
	\varphi(x) \approx \sum_{q=1}^Q \hat \varphi_q  h_q(x),
\]
where $Q\ll\hbox{dim($\myspan{\mathcal F}$)}$.
As said above, the ability for $\mathcal F$ to posses this property is an assumption. It is precisely stated under the notion of small Kolmogorov $n$-width, defined as follows:

Let $\mathcal F$ be a subset of some Banach space $\mathcal X$ and $\VQ$ be a generic $Q$-dimensional subspace of $\mathcal X$. 
The angle between $\mathcal F$ and $\VQ$ is 
\[
	E(\mathcal F;\VQ) := \sup_{\varphi\in \mathcal F} \inf_{v_Q\in \VQ} \|\varphi-v_Q\|_{\mathcal X}.
\]
The Kolmogorov $n$-width of $\mathcal F$ in $\mathcal X$ is given by
\[
	d_Q(\mathcal F,\mathcal X):= \inf\{E(\mathcal F;\VQ)\,|\,\VQ\mbox{ a $Q$-dimensional subspace of } \mathcal X \}.
\]
\noindent The $n$-width of $\mathcal F$ thus measures to what extent the set $\mathcal F$ can be approximated by a $n$-dimensional subspace of $\mathcal X$.

This assumption of small Kolmogorov $n-$width can be taken for granted, but there are also reasons on the elements of $\mathcal F$ that can lead to such a smallness such as regularity of the functions $\varphi\in \mathcal F$. As an example, we can quote, in the periodic settings, the well-known Fourier series. Small truncated Fourier series are good approximations of the full expansion if the decay rate of the Fourier coefficients is fast enough, i.e.\ if the functions $\varphi$ have enough continuous derivatives. In this case, the basis is actually multipurpose since it is not dedicated to the particular set $\mathcal F$. Fourier series are indeed adapted to any set of regular enough functions, the more regular they are, the better the approximation is.
Another property for $\mathcal F$ to have a small Kolmogorov $n$-width is that it satisfies the principle of {\it transform sparsity}, i.e., we assume that the functions $\varphi\in \mathcal F$ are expressed in a {\it sparse} way when written in some orthonormal basis set $\{\psi_i\}$, e.g.\ an orthonormal wavelet basis, a Fourier basis, or a local Fourier basis, depending
on the application: this means that the coefficients $\hat \varphi_i = \langle \varphi, \psi_i\rangle$ satisfy, for some $p$, $0<p<2$, and some $R$:
$$\| \varphi \|_{\ell^p} = \Big(\sum_i |\hat \varphi_i  |^p\Big)^{1/p}\le R.$$
A key implication of this assumption is that if we denote by $\varphi_N$ the sum of the $N$ largest contributions then
$$\exists C(R,p),\ \forall \varphi\in \mathcal F, \quad\|\varphi - \varphi_N\|_{\ell^2} \le C(R,p) (N+1)^{1/2-1/p},$$
i.e.\ there exists a contracted representation of such a $\varphi$. Note that the representation is adaptive and tuned to each $\varphi$ (it is what is called a nonlinear approximation). However, under these assumptions, the theory of compressed sensing (see \cite{donoho2006compressed}), at the price of having a slight degradation of the convergence rate, allows to propose a non-adaptive recovery that is almost optimal. We refer to \cite{donoho2006compressed} and the references therein for more details on this question. Anyway, these are cases where the set of basis functions $\{h_i\}$ does not constitute a multipurpose approximation set, all the contrary: it is tuned to that choice of $\mathcal F$ and will not have any good property for another one.

The difficulty is of course to find the basis set $\{h_i\}$. Note additionally that, from the definition of the small Kolmogorov $n$-width, except in a Hilbertian framework, the optimal elements need not even be in $\myspan{\mathcal F}$.
\bigskip

Let us proceed and propose a way  to better identify the various elements in $\mathcal F$: we consider that they are parametrized with $y \in\Ot \subset \mathbb R^{d_y}$ (with $d_y\ge1$), so that $\mathcal F$ consists of the  parametrized functions $f:\Oo\times\Ot\to \real$. 
In what follows, we denote the function $f$ as a function of $x$ for some fixed parameter value $y$ as $f_y := f(\cdot,y)$.  
However, the role of $x$ and $y$ could be interchanged and both $x$ and $y$ will be considered equally as variables of the same level or as variable and parameter in all what follows.

In this paper, we present a survey of algorithms that search for an affine decomposition of the form
\begin{equation}
	\label{eq:Appr}
	f(x,y) \approx \sum_{q=1}^Q g_q(y) h_q(x).
\end{equation}
We focus on the case where the decomposition is chosen in an optimal way (in terms of sparse representation) and additionally we focus on methods with minimal computational complexity. It is assumed that we have a priori some or all the knowledge on functions $f$ in $\mathcal F$, i.e.\ they are not implicitly defined by a PDE. In that ``implicit'' case there exists a family of reduced modeling approaches such as the reduced basis method; see e.g.~\cite{patera2007reduced}.

Note that the domains $\Oo$ and $\Ot$ can be with finite cardinality $M$ and $N$, in which case the functions can be written as matrices, then, the above algorithms can often be stated as a low-rank approximation: Given a matrix $\rM\in\real^{M\times N}$, find
a decomposition of the matrix $\rM$:
\[
	\rM \approx \rU\,\rV^T
\]
where $\rU$ is of size $M\times Q$ and $\rV$ of size $N\times Q$.

In this completely discrete setting, the Singular Value Decomposition~(SVD), or the related Proper Orthogonal Decomposition~(POD), yields an optimal (in terms of approximability with respect to the~$\|\cdot\|_{\ell^2}$-norm) solution, but is rather expensive to compute. 
After presenting the POD in a general setting in Section~2, we present two alternatives, the Adaptive Cross Approximation~(ACA) in Section~3 and the Empirical Interpolation Method~(EIM), in Section~4, which originate from completely different backgrounds.
We give a comparative overview of features and existing results of those approaches which are computationally much cheaper and yield in practice similar approximation results.
The relation between ACA and the EIM is studied in Section~5. 
Section~6 is devoted to a projection method based on incomplete data known as gappy POD or missing point estimation, which in some cases can be interpreted as an interpolation scheme.

\section{Proper orthogonal decomposition}

Let us start by assuming that we have an unlimited knowledge of the data set and that we have unlimited computer resources --- coming back at the end of this section to more realistic matter of facts. The first approach is known under the generic concept of 
Proper Orthogonal Decomposition (POD) which is a mathematical technique that stands at the intersection of various horizons that have actually been developed independently and concomitantly in various disciplines and is thus known under various names, including:

\begin{itemize}
\item  Proper Orthogonal Decomposition (POD): a term used in turbulence;
\item  Singular Value Decomposition (SVD): a term used in algebra;
\item  Principal Component Analysis (PCA): a term used in statistics  for discrete random processes;
\item  the discrete Karhunen-Loeve transform (KLT): a term used in statistics  for continuous random processes;
\item  the Hotelling transform: a term used in  image processing;
\item  Principal Orthogonal Direction (POD): a  term used in  geophysics;
\item  Empirical Orthogonal Functions  (EOFs): a term used in meteorology and geophysics.
\end{itemize}
All these somewhat equivalent approaches  aim at  obtaining low-dimensional
approximate descriptions of high-dimensional processes, therefore eliminating information which
has little impact on the overall understanding.

\subsection{Historical overview}

As stated above, the POD is present under various forms in many contributions. 

The original SVD was established for real-square matrices in
the 1870's by Beltrami and Jordan, for complex square matrices in 1902 by Autonne, and for
general rectangular matrices in 1936 by Eckart and Young; see also the
generalization to unitarily invariant norms by Mirsky \cite{MR22:5639}. The SVD can be viewed as the
extension of the eigenvalue decomposition for the case of non-symmetric matrices and non-square matrices.

The PCA is a statistical technique. The earliest
descriptions of the technique were given by Pearson~\cite{pearson1901} and Hotelling~\cite{hotelling1933analysis}. The
purpose of the PCA is to identify the dependence structure behind a multivariate stochastic
observation in order to obtain a compact description of it. 

Lumley \cite{lumley2007stochastic} traced the idea of the POD back
to independent investigations by Kosambi \cite{kosambi1943statistics}, Lo\`eve~\cite{loeve1945fonctions}, Karhunen~\cite{karhunen1946spektraltheorie},
Pougachev~\cite{pougachev} and Obukhov~\cite{obukhov}. 

These methods aim at providing a set of orthonormal basis functions that allow to express approximately and optimally any function in the data set. The equivalence between all these approaches has been also investigated by many authors, among them \cite {mees1987singular,kunisch1999control,wu2003note}.

\subsection{Algorithm} Let us now present  the POD algorithm in a semi-discrete framework, that is, we consider a finite family of functions $\{f_{y}\}_{y\in\Ott}$ where $f_y:\Oo\to\real$ for each $y\in\Ot =  \Ott$ where $\Ott$ is finite with cardinality $N$. In this context, the goal is to define an approximation  $P_Q[f_y]$ to $f_y$ defined by 
\begin{equation}\label{POD1}
P_Q[f_y](x) = \sum_{q=1}^Q g_q(y)\, h_q(x)
\end{equation}
with $Q\ll N$.
The POD actually incorporates a scalar product, for functions depending on $x\in\Oo$, and the above projection is then an orthogonal projection on the $Q$-dimensional vectorial space $\myspan{h_q, q=1,\dots,Q}$.

The question is now to select properly the functions $h_q$. With a scalar product, orthonormality is useful, since we would like that these modes are selected in order that they carry as much of the information that exists in the $\{f_{y}\}_{y\in\Ott}$, i.e.\ the first function $h_1$ should be selected such that it provides the best one-term approximation similarly, then $h_q$ should be selected so that, with $h_1, h_2,\dots, h_{q-1}$ it gives the best $q$-term approximation. 
The best $q$-term above is understood in the sense that the mean square error over all $y\in\Ott$ is the smallest.
Such specially ordered orthonormal functions are called the proper orthogonal modes for the function $f(x,y)$. With these functions, the expression (\ref{POD1}) is called the POD of $f$.

\mybox{Proper orthogonal decomposition (POD)}
{
	\begin{itemize}
		\item[\tt a.]	
			Let $\Ott=\{\hat y_1,\ldots,\hat y_N\}$ be a $N$-dimensional dicrete representation of $\Ot$.
		\item[\tt b.]	
			Construct the correlation matrix 
			\[
				\rC_{i,j} =  \tfrac{1}{N}\left(f_{\hat y_j},f_{\hat y_i}\right)_{\Oo},\quad  1\le i,j \le N,
			\]
			where $(\cdot,\cdot)_{\Oo}$ denotes a scalar product of functions depending on $\Oo$.
		\item[\tt c.]	
			Then, solve for the $Q$ largest eigenvalue--eigenvector pairs $(\lambda_q,{\rm v}_q)$ such that
			\begin{equation}
				\label{eq:POD_EigPb}
				 \rC \, \rv_q = \lambda_q \rv_q,\quad  1\le q\le Q.
			\end{equation}
		\item[\tt d.]	
			The orthogonal POD basis functions $\{h_1,\ldots,h_Q\}$ such that $\Vpod=\myspan{h_1,\ldots,h_Q}$ are then given by the 			linear combinations
			\[
				h_q(x) = \sum_{n=1}^{N} (\rv_q)_n \,  f(x,\hat y_n),\quad 1\le q\le Q,\quad x\in\Oo,
			\]
			and where $(\rv_q)_n$ denotes the $n$-th coefficient of the eigenvector $\rv_q$. 
	\end{itemize}
	\vspace{6pt}
	{\bf Approximation:}   
	The approximation $P_Q[f_y]$ to $f_y:\Oo\to\real$, for any $y\in\Ot$, is then given by
\[
	P_Q[f_y] (x) = \sum_{q=1}^Q g_q(y)\,h_q(x),\quad x\in\Oo,
\] 
with $g_q(y) = \frac{ (f_y,h_q)_{\Oo} }{ (h_q,h_q)_{\Oo} }$.
}

\begin{prop}
	The approximation error 
	\[
		d_2^{\textnormal{POD}}(Q) = \sqrt{ \frac{1}{N}\sum_{y \in\Ott}\| f_y - P_Q[f_y]\|^2_{\Oo} }
	\]
	minimizes the mean square error  $\sqrt{ \frac{1}{N}\sum_{y \in\Ott}\| f_y - {\mathcal P}_Q[f_y]\|^2_{\Oo} }$ over all projection operators ${\mathcal P}_Q $ onto a space of dimension $Q$. It
	is given by
	\begin{equation}
		\label{eq:POD_err}
		d_2^{\textnormal{POD}}(Q) = \sqrt{\sum_{q=Q+1}^N \lambda_q},
	\end{equation}
	where $\{\lambda_{Q+1},\ldots,\lambda_N\}$ denotes the set of the $N-Q$ smallest eigenvalues of the eigenvalue problem (\ref{eq:POD_EigPb}).
\end{prop}

\begin{remark}[Relation to SVD]
	\label{rem:SVD}
		If the scalar product $(\cdot,\cdot)_{\Oo}$ is approximated in the sense of $\ell^2$ on a discrete set of points $\Oot=\{\hat x_1,\ldots,\hat x_M\}\subset \Oo$, i.e.
		\[
			(v,w)_{\Oot} = \frac{|\Oo|}{M}\sum_{i=1}^M v(\hat x_i)\,w(\hat x_i),
		\]
		then we see that $\rC=\rA^T\, \rA$ where $\rA$ is the matrix defined by $\rA_{i,j} =\tfrac{|\Oo|}{\sqrt{N}} f_{y_j}(\hat x_i)$. And thus, the square roots of the eigenvalues \eqref{eq:POD_EigPb} are singular values of $\rA$.
\end{remark}

\begin{remark}[Infinite dimensional version]
In the case where the POD is processed by leaving the parameter $y$
continuous in $\Ot$, the correlation matrix becomes an operator
$C:L^2(\Ot)\to L^2(\Ot)$ with
kernel $C(y_1, y_2) = \left(f_{y_1},f_{y_2}\right)_{\Oo}$ that acts on functions of $y\in\Ot$ as follows
\[
	(C\phi)(y) = \left(C(y,\cdot),\phi\right)_{\Ot}, \quad
        \phi\in L^2(\Ot).
\]
Assuming that $f\in L^2(\Oo\times \Ot)$, by the results obtained in
\cite{MR1158733} (that generalize Mercer's theorem to more general domains) there exists a sequence of positive real eigenvalues (that can be ranked in decreasing order)
and associated orthonormal eigenvectors, which can be used to
construct best $L^2$-approximations~\eqref{eq:Appr}.
\end{remark}

The infinite dimensional version is important to understand the generality of the approach, e.g.\ how the various POD algorithms are linked together. In essence, this boils down to spectral theory of self-adjoint operators, either finite (in the matrix case) or infinite (for integral operator defined with symmetric kernels). Such operators have positive real eigenvalues and the corresponding eigenvectors can be ranked in decreasing order of eigenvalues. The approximation is based on considering the only eigenmodes that corresponds to the largest eigenvalues, they are those that carry the maximum information.

In practice though,  both in the $x$ and the $y$ variables, sample sets $\Oot$ and $\Ott$ are devised. 
Depending on the size of $N$, the solution of the eigenvalue problem
\eqref{eq:POD_EigPb} can be prohibitively expensive. Most of the time though, there is not much hint on the way these training points should be chosen and they are generally quite large sets with $N\gg Q$.

We finally remind that the original goal is to approximate any function $f(x,y)$ for all $x\in\Oo$ and $y\in\Ot$. 
In this regard, the error bound (\ref{eq:POD_err}) only provides an upper error estimate for functions $f_{y}$ with $y\in\Ott$
and no certified error bound for functions $f_{y}$ with $y\in\Ot\backslash\Ott$ can be provided. 

\section{Adaptive Cross Approximation}
 In order to cope with the difficulty of implementation of the POD algorithms, let us present here the {\it Adaptive Cross Approximation}. The approximation leading to~\eqref{eq:Appr} is
\begin{equation}\label{eq:CA}
f(x,y)\approx \mathfrak{I}_Q[f_y](x):=\begin{bmatrix} f(x,y_1)\\\vdots\\ f(x,y_Q)\end{bmatrix}^T
\rM_Q^{-1}
\begin{bmatrix} f(x_1,y)\\\vdots\\ f(x_Q,y)\end{bmatrix}
\end{equation}
with points $x_q$, $y_q$, $q=1,\dots,Q$, chosen such that the matrix
\[
\rM_Q:=\begin{bmatrix} f(x_1,y_1) & \ldots & f(x_1,y_Q)\\
\vdots & &\vdots\\
f(x_Q,y_1) & \ldots & f(x_Q,y_Q)
\end{bmatrix}\in\real^{Q\times Q}
\]
is invertible. Notice that while $P_Q$ used in the construction of the
POD is an orthogonal projector, $\mathfrak{I}_Q:C^0(\Oo)\to
\VQ$ is an interpolation
operator from the space of continuous functions $C^0(\Oo)$ onto the system $\VQ:=\myspan{f_{y_1},\dots,f_{y_Q}}$, i.e.
\[\mathfrak{I}_Q[f_y](x_q)=f(x_q,y)\quad \text{for all }y\text{ and
}q=1,\dots,Q.\]
Due to the symmetry of $x$ and $y$ in \eqref{eq:CA}, we also have
$\mathfrak{I}_Q[f_{y_q}](x)=f(x,y_q)$ for all $x$ and $q=1,\dots,Q$.

\subsection{Historical overview}
Approximations of type \eqref{eq:CA} were first considered by Micchelli and
Pinkus in~\cite{MR510921}. There, it was proved for so-called
totally positive functions~$f$, i.e.\  continuous functions
${f:[0,1]\times[0,1]\to\real}$ with non-negative determinants
\[
\left| \begin{bmatrix} f(\xi_1,\upsilon_1) & \ldots & f(\xi_1,\upsilon_q)\\
\vdots & & \vdots\\
f(\xi_q,\upsilon_1) & \ldots & f(\xi_q,\upsilon_q)
\end{bmatrix}\right|
\]
for all $0\leq \xi_1<\ldots<\xi_q\leq 1$, $0\leq \upsilon_1<\ldots<\upsilon_q\leq 1$,
and $q=1,\dots,Q$, that such approximations are
optimal with respect to the $L^1$-norm, i.e.
\[
\min_{u_q,v_q} \int_0^1\int_0^1 \left| f(x,y)-\sum_{q=1}^Q u_q(x)v_q(y) \right|\ud y\ud x
=\int_0^1\int_0^1 \left| f(x,y)-\mathfrak{I}_Q[f_y](x) \right|\ud y\ud x,
\]
where $\mathfrak{I}_Q$ is defined at implicitly known nodes
$x_1,\dots,x_Q$ and $y_1,\dots,y_Q$; see \cite{MR510921} for an
additional technical assumption.

Instead of $L^1$-estimates, it is usually required to obtain $L^\infty$-estimates.
The obvious estimate 
\[
	\|f_y-\mathfrak{I}_Q[f_y]\|_{L^\infty(\Oo)}\leq(1+\sigma_1[f])\inf_{v\in\VQ}\|f_y-v\|_{L^\infty(\Oo)}
\]
contains the expression
\[
	\sigma_1[f]:= \sup_{x\in \Oo} \Big\| \rM_Q^{-T}\begin{bmatrix} f(x,y_1)\\\vdots\\ f(x,y_Q)\end{bmatrix} \Big\|_{\ell^1}.
\]
Since there is usually no estimate on the previous infimum (note that
$\VQ$ also depends on $\mathcal F =\{f_y\}_{y\in \Ot}$), one tries to relate $f_y-\mathfrak{I}_Q[f_y]$ with the interpolation error
in another system $\WQ=\myspan{w_1,\dots,w_Q}$ of functions (e.g.\
polynomials, spherical harmonics, etc.); cf.~\cite{MR2001j:65022,MR2004a:65177}. Assume that the
determinant of the Vandermonde matrix ${\rm W}_Q:=[w_i(x_j)]_{i,j=1,\dots,Q}$ does not vanish and let
$L:\Oo\to\real^Q$ be the vector consisting of Lagrange
functions $L_i\in \WQ$, i.e.\
$L_i(x_j)=\delta_{ij}$, $i,j=1,\dots,Q$. Then, the interpolation
operator~$\mathfrak{I}_Q'$ defined over $C^0(\Oo)$ with values in~$\WQ$ can be represented as
\[
\mathfrak{I}_Q'[ \varphi](x)=\begin{bmatrix}  \varphi(x_1)\\\vdots\\
\varphi(x_Q)\end{bmatrix}^TL(x),\quad \varphi\in C^0(\Oo),
\]
and we obtain
\begin{align*}
f_y(x)-\mathfrak{I}_Q[f_y](x)&=f_y(x)-\begin{bmatrix} f(x_1,y)\\
  \vdots\\f(x_Q,y)\end{bmatrix}^TL(x)
-\left(\begin{bmatrix} f(x,y_1)\\ \vdots \\ f(x,y_Q)\end{bmatrix}-\rM_Q^TL(x)\right)^T\rM_Q^{-1} \begin{bmatrix} f(x_1,y)\\
  \vdots\\f(x_Q,y)\end{bmatrix}\\
&=f_y(x)-\mathfrak{I}_Q'[f_y](x) -\begin{bmatrix} f_{y_1}(x)-\mathfrak{I}_Q'[f_{y_1}](x)\\ \vdots \\ f_{y_Q}(x) -\mathfrak{I}_Q'[f_{y_Q}](x)\end{bmatrix}^T\rM_Q^{-1} \begin{bmatrix} f(x_1,y)\\
  \vdots\\f(x_Q,y)\end{bmatrix}.
\end{align*}
Hence, for any $y\in\Ot$
\begin{equation}\label{eq:errest}
\|f_y-\mathfrak{I}_Q[f_y]\|_{L^\infty(\Oo)}\leq (1+\sigma_2[f])\max_{z\in\{y,y_1,\dots,y_Q\}}\|f_z-\mathfrak{I}_Q'[f_z]\|_{L^\infty(\Oo)},
\end{equation}
where
\[
\sigma_2[f]:= \sup_{y\in \Ot} \Big\|\rM_Q^{-1}\begin{bmatrix} f(x_1,y)\\\vdots\\ f(x_Q,y)\end{bmatrix} \Big\|_{\ell^1}.
\]

\subsection{Construction of interpolation nodes}\label{ssec:constrip}
The assumption that the determinant of the Vandermonde matrix~${\rm W}_Q$ does not vanish,
can be guaranteed by the choice of $x_1,\dots,x_Q$.
To this end, let $Q$ linearly independent functions
$w_1,\dots,w_Q$ be given as above. As in \cite{Beb10}, we construct linearly independent
functions~$\ell_1,\dots,\ell_Q$ satisfying ${\ell_q(x_p)=0}$, $p<q$, and
$\myspan{\ell_1,\dots,\ell_q}=\WQ$,
$q\leq Q$, in the following way. Let $\ell_1=w_1$ and $x_1\in
\Oo$ be a maximum of $|\ell_1|$. Assume that
$\ell_{Q-1}$ has already been constructed. For the construction of
$\ell_Q$ define $\ell_{Q,0}:=w_Q$ and
\[
  \ell_{Q,q}:=\ell_{Q,q-1}-\ell_{Q,q-1}(x_q)  \frac{\ell_q}
  {\ell_q(x_q)} ,\quad q=1,\dots,Q-1.
\]
Then $\ell_{Q,Q-1}(x_q)=0$, $q<Q$, and
$\myspan{\ell_{Q,0},\dots,\ell_{Q,Q-1}}
=\myspan{\ell_1,\dots,\ell_{Q-1},w_Q}$.
Hence, we set $\ell_Q:=\ell_{Q,Q-1}$ and choose
\begin{equation}\label{eq:choicex}
x_Q:=\argsup{x\in\Oo} |\ell_Q(x)|.
\end{equation}
The previous construction guarantees unisolvency at the nodes $x_q$, $q=1,\dots,Q$.
\begin{lem} It holds that $\det {\rm W}_Q\neq0$.
\end{lem}
\begin{proof}
Since $\myspan{\ell_1,\dots,\ell_Q}=\myspan{w_1,\dots,w_Q}$ it
follows that there is a non-singular matrix ${\rm T}\in\real^{Q\times Q}$ such that
\[
\begin{bmatrix}
\ell_1\\\vdots\\\ell_Q
\end{bmatrix}
={\rm T} \begin{bmatrix}
w_1\\\vdots\\ w_Q
\end{bmatrix}.
\]
Hence, ${\rm R}_Q={\rm T}\,{\rm W}_Q$ where ${\rm R}_Q:=[\ell_i(x_j)]_{i,j=1}^Q$ is upper triangular. The assertion
follows from \[\det {\rm R}_Q=\ell_1(x_1)\cdot\ldots\cdot \ell_Q(x_Q)
\neq0.\]
\end{proof}
As an example, we choose $\WQ=\Pi_{Q-1}$ the space of polynomials of
degree at most $Q-1$. Then, it follows from \eqref{eq:errest} that ACA
converges if, e.g.,
$f$ is analytic with respect to~$x$, and the speed of convergence is
determined by the decay of $f$'s derivatives or the elliptical radius
of the ellipse in which $f$ has a holomorphic extension. 
Furthermore, it can be seen that 
\[
	\ell_Q(x)=\prod_{q=1}^{Q-1} (x-x_q).
\]
Hence, the choice~\eqref{eq:choicex} of $x_Q$ is a generalization of a construction that
is due to Leja~\cite{MR0100726}. Leja recursively defines a sequence
of nodes $\{x_1,\dots,x_Q\}$ for polynomial interpolation
in a compact set~$K\subset\mathbb{C}$ as follows. Let $x_1\in K$ be
arbitrary. Once $x_1,\dots,x_{Q-1}$ have been found, choose
$x_Q\in K$ so that
\[
\prod_{q=1}^{Q-1}|x_Q-x_q|=\max_{x\in K}\prod_{q=1}^{Q-1}|x-x_q|.
\]
In~\cite{Tay08} it is proved that Lebesgue constants associated
with Leja points are subexponential for fairly general compact sets
in~$\mathbb{C}$; see also~\cite{MR1039671}.
Hence, analyticity is required in general for the convergence of the
interpolation process.

The expression $\sigma_2[f]$ on the right-hand side of
\eqref{eq:errest} can be controlled by the choice of the points
$y_1,\dots,y_Q\in\Ot$. Due to Laplace's theorem 
\[
\left(\rM_Q^{-1}\begin{bmatrix} f(x_1,y)\\\vdots\\
    f(x_Q,y)\end{bmatrix}\right)_q=\frac{\det \rM_q(y)}{\det \rM_Q},\quad q=1,\dots,Q,
\]
where $\rM_q(y)$ arises from replacing the $q$-th column of $\rM_Q$ by the
vector $[f(x_1,y),\dots,f(x_Q,y)]^T$, we obtain that $\sigma_2[f]\leq
Q$ if $y_1,\dots,y_Q$ are chosen such that
\begin{equation}\label{eq:maxvol}
|\det \rM_Q|\geq |\det \rM_q(y)|,\quad q=1,\dots,Q,\;y\in\Ot.
\end{equation}
In connection with the so-called {\it maximum volume condition} \eqref{eq:maxvol},
we also refer to the error estimates in \cite{JSchn09} which are based on the
technique of \emph{exact annihilators} (see~\cite{MR1830958,MR1019046}) in order to provide similar results as~\eqref{eq:errest}.

\subsection{Incremental construction}
The maximum volume condition~\eqref{eq:maxvol} is
difficult to satisfy by an a priori choice of
$y_1,\dots,y_Q$. Therefore, the following incremental construction  of
approximations \eqref{eq:CA},
which is called {\it adaptive cross approximation}~(ACA)~\cite{MR2001j:65022}, has turned out to be practically more
relevant. Let $r_0(x,y):=f(x,y)$ and define the sequence of remainders as
\begin{equation}\label{eq:resid}
r_q(x,y):=r_{q-1}(x,y)-\frac{r_{q-1}(x,y_q)\,r_{q-1}(x_q,y)}{r_{q-1}(x_q,y_q)},\quad q=1,\dots,Q,
\end{equation}
where $x_q$ and $y_q$ are chosen such that $r_{q-1}(x_q,y_q)\neq0$. Then, the algorithm can be summarized as follows.
\mybox{Bivariate Adaptive Cross Approximation (ACA2)}
{
Set $q:=1$.\\
		While $\err < \tol$
		\begin{itemize}
				\item[\tt a.]	
				Define the remainder $r_{q-1} =  f-\sum_{i=1}^{q-1} c_i$ and choose $(x_q,y_q)\in\Oo\times\Ot$ such that
				\[
					r_{q-1}(x_q,y_q)\neq0.
				\]
				\item[\tt b.]	
				Define the next tensor product by
				\[
					c_q(x,y) = 
					\frac{r_{q-1}(x,y_q) \, r_{q-1}(x_q,y) }{r_{q-1}(x_q,y_q)}.
				\]				
				\item[\tt c.]	
				Define the error level by
				\[
					\err = \|r_{q-1}\|_{L^\infty(\Oo\times\Ot)}
				\]
				and set $q:=q+1$.
		\end{itemize}
}

Since $r_{q-1}(x_q,y_q)$ coincides with the $q$-th diagonal entry of the upper triangular factor of the LU decomposition of
$\rM_Q$, we obtain that $\det \rM_Q\neq0$. In \cite{MR2004a:65177}, it
is shown that 
\begin{equation}
	\label{eq:rel_f_Int}
	f(x,y)=\mathfrak{I}_Q[f_y](x)+r_Q(x,y)
\end{equation}
and
\[
\mathfrak{I}_Q[f_y](x)=\sum_{q=1}^Q
r_{q-1}(x,y_q)\frac{ r_{q-1}(x_q,y)}{r_{q-1}(x_q,y_q)}.
\]
This method is used in
\cite{MR2280533} (see also \cite{Chapman}) under the name
{\it Geddes-Newton series expansion} for the numerical integration of bivariate
functions, where instead of the maximum volume
condition~\eqref{eq:maxvol} $(x_q,y_q)$ is found from maximizing
$|r_{q-1}|$. This choice of $(x_q,y_q)$ is usually referred to as {\it global pivoting}.
Another pivoting strategy is the so-called {\it partial
  pivoting}, i.e.,  $y_q$ is chosen in the $q$-th step such that
\[
|r_{q-1}(x_q,y_q)|\geq |r_{q-1}(x_q,y)|\text{ for all }y\in\Ot
\]
for $x_q\in\Oo$ chosen by \eqref{eq:choicex}.
For the latter condition (and in particular for the stronger global pivoting) the conservative bound $\sigma_2[f]\leq 2^Q-1$ can be guaranteed;
see \cite{MR2001j:65022}. The actual growth of $\sigma_2[f]$ with respect
to~$Q$ is, however, typically significantly weaker.

\subsection{Application to matrices}
Approximations of the form~\eqref{eq:CA} are particularly useful when
they are applied to large-scale matrices $\rA\in\real^{M\times N}$. In this case,
\eqref{eq:CA} becomes
\begin{equation}\label{eq:maca}
\rA\approx \tilde{\rA}:=\rA_{:,\sigma}\,\rA_{\tau,\sigma}^{-1}\,\rA_{\tau,:},
\end{equation}
where $\tau:=\{i_1,\dots,i_Q\}$ and $\sigma:=\{j_1,\dots,j_Q\}$ are
sets of row and column indices, respectively, such that
$\rA_{\tau,\sigma}\in\real^{Q\times Q}$ is invertible.
Here and in the following, we use the notation $\rA_{\tau,:}$ for the
rows $\tau$ and $\rA_{:,\sigma}$ for the columns $\sigma$ of $\rA$.
Notice that the approximation~$\tilde{\rA}$ has rank at most~$Q$ and 
is constructed from few of the original matrix
entries. Such kind of approximations  were investigated by
Eisenstat and Gu~\cite{MR97h:65053} and Tyrtyshnikov~et~al.~\cite{MR99d:15015}
in the context of the maximum volume condition.
Again, the approximation can be constructed incrementally by the
sequence of remainders $\rR^{(0)}:=\rA$ and
\[
\rR^{(q)}:=\rR^{(q-1)}-\frac{\rR^{(q-1)}_{:,j_q} \, \rR^{(q-1)}_{i_q,:}}{\rR^{(q-1)}_{i_q, j_q}},\quad q=1,\dots,Q,
\]
where the index pair $(i_q,j_q)$ is chosen such that $\rR^{(q-1)}_{i_q j_q}\neq0$. The previous
condition guarantees that $\rA_{\tau,\sigma}$ is invertible, and we obtain
\[
\tilde \rA= \sum_{q=1}^Q \frac{\rR^{(q-1)}_{:,j_q}\, \rR^{(q-1)}_{i_q,:}}{\rR^{(q-1)}_{i_q, j_q}}.
\]
If $\rA$ arises from evaluating a smooth function at given points,
then $\rR^{(q)}$ can be estimated using \eqref{eq:errest}.

In order to avoid the computation of each entry of the remainders~$\rR^{(q)}$, it is
important to notice that only the entries in the $i_q$-th row and
the $j_q$-th column of $\rR^{(q-1)}$ are required for the construction
of~$\tilde \rA$. Therefore, the following algorithm computes the column vectors~${\ru_q:=\rR^{(q-1)}_{:,j_q}}$ and row vectors~${\rv_q:=\rR^{(q-1)}_{i_q,:}}$ resulting in
\begin{equation}\label{eq:Auv}
	\tilde \rA= \sum_{q=1}^Q \frac{\ru_q\,\rv_q^T}{(\rv_q)_{j_q}}.
\end{equation}
The iteration stops after~$Q$ steps if the error satisfies
\begin{equation}\label{eq:ACAstop}
\|\rA-\tilde\rA\|_{\ell^2}=\|\rR^{(Q)}\|_{\ell^2}<\varepsilon
\end{equation}
with given accuracy~$\varepsilon>0$. The previous condition
cannot be evaluated with linear complexity. Since the next rank-$1$
term $(\rv_{Q+1})_{j_{Q+1}}^{-1}\ru_{Q+1}\rv_{Q+1}^T$
approximates $\rR^{(Q)}$, we replace \eqref{eq:ACAstop} with the error indicator
\[
\frac{\|\ru_{Q+1}\rv_{Q+1}^T\|_{\ell^2}}{|(\rv_{Q+1})_{j_{Q+1}}|}=\frac{\|\ru_{Q+1}\|_{\ell^2}\|\rv_{Q+1}\|_{\ell^2}}{|(\rv_{Q+1})_{j_{Q+1}}|}<\varepsilon.
\]

\mybox{Adaptive Cross Matrix Approximation}
{
Set $q:=1$.\\
		While $\err < \tol$
		\begin{itemize}
				\item[\tt a.]
Choose $i_q$ such that
	\[
					\rv_q:=\rA^T_{i_q,:}-\sum_{\ell=1}^{q-1}\frac{(\ru_\ell)_{i_q}}{(\rv_\ell)_{j_\ell}} \rv_\ell
				\]
is nonzero and $j_q$ such that $|(\rv_q)_{j_q}|=\max_{j=1,\dots,N} |(\rv_q)_j|$.
				\item[\tt b.]	
				Compute the vector
\[
\ru_q:=\rA_{:,j_q}-\sum_{\ell=1}^{q-1}
\frac{(\rv_\ell)_{j_q}}{(\rv_\ell)_{j_\ell}} \ru_\ell.
\]	
				\item[\tt c.]	
				Compute the error indicator
				\[
					\err = |(\rv_q)_{j_q}|^{-1}\|\ru_q\|_{\ell^2}\|\rv_q\|_{\ell^2}
				\]
				and set $q:=q+1$.
		\end{itemize}
}

\begin{remark}
Notice that almost no condition has been imposed on the row
index~$i_q$. The following three methods are commonly used to choose~$i_q$.
In addition to choosing $i_q$ randomly, $i_q$ can be found as
\[
i_q:=\argmax{i=1,\dots,M} |({\rm u}_{q-1})_i|,
\]
which leads to a cyclic pivoting strategy. If $\rA$ stems from the
evaluation of a function at given nodes, then the construction of
Section~\ref{ssec:constrip} should be used in order to guarantee the
well-posedness of the interpolation operator $\mathfrak{I}_Q'$ and
exploit the error estimate~\eqref{eq:errest}. 

In some cases (see \cite{BOGR04}), it is required to
put more effort in the choice of~$i_q$ to guarantee a well-suited
approximation space~$\myspan{\rA_{i_1,:},\dots,\rA_{i_Q,:}}$;
cf.~\cite{Bebendorf:2008}.
\end{remark}

Instead of the $M\cdot N$ entries of $\rA$, we only have to
compute $Q(M+N)$ entries of $\rA$ for the approximation by $\tilde {\rm A}$. The construction of~\eqref{eq:Auv}
requires $\mathcal{O}(Q^2(M+N))$ arithmetic operations, and $\tilde{\rA}$ can be
stored with $Q(M+N)$ units of storage.
Possible redundancies among the vectors $\ru_q$, $\rv_q$, $q=1,\dots,Q$,
can be removed via orthogonalization.

The origin of this matrix version of ACA is the construction of so-called
hierarchical matrices~\cite{MR2000c:65039,MR2001i:65053,Bebendorf:2008} for the efficient treatment of
integral formulations of elliptic boundary value
problems. Hierarchical matrices allow to treat discretizations of such non-local
operators with logarithmic-linear complexity. To this end, subblocks $\rA_{t,s}$ from a suitable partition of
large-scale matrices $\rA$ are approximated by low-rank matrices.

A form that is slightly different from~\eqref{eq:maca} and which looks
more complicated at first glance is
\[
\rA_{t,s}\approx \hat{\rA}_{t,s}:=\rA_{:,\sigma_t}\,\rA_{\tau_t,\sigma_t}^{-1}\,\rA_{\tau_t,\sigma_s}\,\rA_{\tau_s,\sigma_s}^{-1}\,\rA_{\tau_s,:}
\]
with suitable index sets $\tau_t$, $\sigma_t$, $\tau_s$, and $\sigma_s$
depending on the respective index $t$ or $s$ only. Notice that in
contrast to $\tilde \rA$, $\hat\rA$ 
does not interpolate $\rA$ on the ``cross'' but rather at
single points specified by the indices~$\tau_t,\sigma_s$, i.e.\ $\hat \rA_{\tau_t,\sigma_s}=\rA_{\tau_t,\sigma_s}$. 
The advantage of this
approach is the fact that the large parts
$\rA_{:\sigma_t}\rA_{\tau_t,\sigma_t}^{-1}$ and
$\rA_{\tau_s,\sigma_s}^{-1}\rA_{\tau_s,:}$ depend only on either one of the two index
sets $t$ or $s$, while only the small matrix $\rA_{\tau_t,\sigma_s}$
depends on both. This allows to further reduce the complexity of
hierarchical matrix approximations by constructing so-called nested
bases approximations~\cite{MBRV11},
which are mandatory to efficiently treat high-frequency Helmholtz
problems; see~\cite{BKV12}.

\subsection{Relation with Gaussian elimination}
Without loss of generality, we may assume for the moment that 
$i_q=j_q=q$, $q=1,\dots,Q$.
Otherwise, interchange the rows and columns of the original matrix~$R^{(0)}$.
Then 
\[
	\rR^{(q)} = \left(\rI- \frac{\rR^{(q-1)} \re_q \re_q^T}{\re_q^T\rR^{(q-1)}\re_q}\right)\rR^{(q-1)} = \rL^{(q)} \rR^{(q-1)},
\]
where $\rL^{(q)}\in\real^{m\times n}$ is the matrix
\[
\rL^{(q)} =\begin{bmatrix}
1 &        &   &   & & & \\
  & \ddots &   &   & & & \\
  &        & 1 &   & & & \\ \\
  &        &   & 0 & & & \\ \\
  &        &   & -\frac{\re_{q+1}^T \rR^{(q-1)} \re_q}{\re_q^T \rR^{(q-1)} \re_q} & 1 & & \\
  &        &   & \vdots &   & \ddots & \\
  &        &   & -\frac{\re_M^T \rR^{(q-1)} \re_q}{\re_q^T \rR^{(q-1)} \re_q} &   & & 1
\end{bmatrix},
\]
which differs from a Gaussian matrix only in the position $(q,q)$; cf.~\cite{MR2001j:65022}.
This relation was exploited in~\cite{HH11} for the convergence analysis of 
ACA in the case of positive definite matrices~$\rm A$.

Furthermore, it is an interesting observation that ACA reduces the
rank of the remainder in each step, i.e.\
$\textnormal{rank}\,{\rm R}^{(q)}=\textnormal{rank}\, {\rm R}^{(q-1)}-1$.  This was first discovered by Wedderburn in \cite[p.\
69]{MR0168568}; see also~\cite{ChFuGo1995,MR2001j:65022}.
Hence, ACA may be regarded as a rank revealing LU factorization~\cite{MR85i:65037,MR93i:15016}.
As we know, it is possible that the elements 
grow in the LU decomposition algorithm; cf.~\cite{MR97g:65006}.
Thus the exponential bound $2^Q$ on $\sigma_2[f]$ is not a result of 
overestimation.

\subsection{Generalizations of ACA}
The adaptive cross approximation can easily be generalized to a linear
functional setting. Instead the evaluation of the remainders at the
chosen points $x_q$, $y_q$, $q=1,\dots,Q$, one considers the 
recursive construction
\[
r_q(x,y):=r_{q-1}(x,y)-\frac{\langle
  r_{q-1}(x,\cdot),\psi_q\rangle\,\langle
  \varphi_q,r_{q-1}(\cdot,y)\rangle}
{\langle \varphi_q,r_{q-1},\psi_q\rangle},\quad q=1,\dots,Q.
\]
Here, $\varphi_q$ and $\psi_q$ denote given linear functionals acting on~$x$ and~$y$, respectively. 
It is easy to show (see~\cite{MBCK11}) that
\begin{equation}\label{eq:zeros}
\langle
\varphi_i,r_q(\cdot,y)\rangle=0
=\langle r_q(x,\cdot),\psi_i\rangle\quad
\text{for all }i\leq q,\, x\in\Oo\text{ and }y\in\Ot.
\end{equation}
Hence, $r_q$ vanishes for an increasing number of functionals and 
\[
\mathfrak{I}''_Q[f_y](x):=\sum_{q=1}^Q \langle
r_{q-1}(x,\cdot),\psi_q\rangle\frac{\langle \varphi_q,r_{q-1}(\cdot,y)\rangle}{\langle
\varphi_q,r_{q-1},\psi_q\rangle}
\]
gradually interpolates $f_y$ (in the sense of functionals). The adaptive cross
approximation~\eqref{eq:resid} is obtained from choosing the Dirac functionals
$\varphi_q:=\delta_{x_q}$ and $\psi_q:=\delta_{y_q}$.

The benefits of the separation of variables resulting
form~\eqref{eq:CA} are even more important for
multivariate functions~$f$. We present two ways to generalize
\eqref{eq:resid} to functions depending on $d$~variables. An obvious idea is to group the set of
variables into two parts each containing $d/2$~variables; see \cite{MBCK11} for a method that uses the
covariance of~$f$ to construct this separation. Each of the two parts can be
treated as a single new variable. Then, the application of \eqref{eq:resid} results
in a sequence of less-dimensional functions which inherit the
smoothness of~$f$. Hence, \eqref{eq:resid} can be applied again until
only univariate functions are left. Due the nestedness of the construction, the
constructed approximation cannot be regarded as an
interpolation. Error estimates for this approximation were derived in
\cite{Beb10} for $d=3,4$. The application to tensors of order $d>2$
was presented in~\cite{1572446,TT-cross,BGK10}. 

A more sophisticated way to generalize ACA to multivariate functions
is presented in \cite{MBAK11}. For the case $d=3$, the sequence of
remainders is constructed as
\[
r_q(x,y,z):=r_{q-1}(x,y,z)-\frac{r_{q-1}(x,y,z_q)\,r_{q-1}(x,y_q,z)\,r_{q-1}(x_q,y,z)\,r_{q-1}(x_q,y_q,z_q)}
{r_{q-1}(x,y_q,z_q)\,r_{q-1}(x_q,y,z_q)\,r_{q-1}(x_q,y_q,z)}
\]
instead of \eqref{eq:resid}. Notice that this kind of approximation
requires that $x_q,y_q,z_q$ can be found such that the denominator
$r_{q-1}(x,y_q,z_q)\,r_{q-1}(x_q,y,z_q)\,r_{q-1}(x_q,y_q,z)\neq0$.
On the other hand, the advantage of this generalization is that it is equi-directional
in contrast to the aforementioned idea, i.e., none of the variables is
preferred to the others. Hence, similar to~\eqref{eq:zeros} we obtain
for all $x,y,z$
\[
r_q(x,y,z_i)=r_q(x,y_i,z)=r_q(x_i,y,z)=0,\quad i\leq q.
\]

\section{Empirical Interpolation Method}
\subsection{Historical overview} 

The Empirical Interpolation Method (EIM) \cite{Barrault:2004wb} originates from reduced order modeling and its application to the resolution of parameter dependent partial differential equations. We are thus in the context where the set of solutions $u(\cdot,y)$ to the PDE generates a manifold, parametrized by $y$ (the parameter is generally called $\mu$ in these applications) that possesses a small Kolmogorov $n$-width. 
In the construction stage of the reduced basis method, the reduced basis is constructed from a greedy approach where each new basis function, that is a solution to the PDE associated to an optimally chosen parameter, is incorporated recursively.
The selection criteria of the parameter is based on maximal  (a posteriori) error estimates over the parameter space.
This construction stage can be expensive: indeed it requires an initial accurate classical discretization method of finite element, spectral or finite volume type and every solution associated to a parameter that is optimally selected, needs to be approximated during this stage by the classical method. 
Once the preliminary stage is performed off-line, all the approximations of solutions corresponding to a new parameter are performed as a linear combination of the (few) basis functions
constructed during the first phase. This second on-line stage is very cheap. This is due to two facts. The first one is related to the fact that  the greedy approach is proven to be quite optimal \cite{buffa2012priori, binev2011convergence, devore2012greedy}, for exponential or polynomial decay of the Kolmogorov $n$-width, the greedy method provides a basis set that has the same feature.

The second fact is related to the approximation process. A Galerkin approximation in this reduced space indeed provides very good approximations, and if $Q$ modes are used, a linear PDE can be simulated by inverting $Q\times Q$ matrices only, i.e.\ much smaller complexity than the classical approaches.

In order that the same remains true for nonlinear PDE's, a strategy, similar to the pseudo-spectral approximation for high-order Fourier or polynomial approximations has been sought. This involves the use of an interpolation operator. In order to be coherent, an approximation $u_Q(\cdot,y) = \sum_{i=1}^Q \alpha_i(y)\, u(\cdot,y_i) $ being given (where the $y_i$ are the parameters that define the reduced basis snapshots) we want to approximate ${\mathcal G}(u_Q(\cdot,y))$ (${\mathcal G}$ being a nonlinear functional)  as a linear combination 
$${\mathcal G}(u_Q(\cdot,y)) \approx \sum_{i=1}^Q \beta_i(y) {\mathcal G}(u(\cdot,y_i)).$$
The derivation of the set $\{\beta_i\}_i$ from $\{\alpha_i\}_i$ needs to be very fast, it is defined by interpolation through the Empirical Interpolation Method defined in the following section. This has been extensively used for different types of equations in \cite{grepl2007efficient} and has led to the definition of general interpolation techniques and rapid derivation of the associated points.

The approach having a broader scope than only the use in reduced basis approximation, a dedicated analysis of the approximation properties for sets with small-Kolmogorov $n$-width has been presented in \cite{Maday:2009tg}. This approach for nonlinear problems has actually also been used for problems where the dependency in the parameter is involved (the so called ``non-affine problems'') and has boosted the domain of application of reduced order approximations.

\subsection{Motivation}
As said above and in the introduction, we are in a situation where the set $\mathcal F =\{f(\cdot,y)\}_{y\in \Ot}$ denotes a family of parametrized functions with small Kolmogorov $n$-width. 
We therefore do not identify $\Oo$ with $\Ot$. 
In addition, for a given parameter $y$, $f(\cdot,y)$ is supposed to be accessible at all values in $\Oo$.

The EIM is designed to find approximations to members of $\mathcal F$ through an interpolation operator~$\Inter_q$ that interpolates the function $f_y = f(\cdot,y)$ 
at some particular points in $\Oo$.
That is, given an interpolatory system defined by a set of basis functions $\{h_1,\ldots,h_q\}$ (linear combination of particular ``snapshots'' $f_{y_1},\ldots,f_{y_q}$) and interpolation points $\{x_1,\ldots,x_q\}$, the interpolant~$\Inter_q[f_y]$ of $f_y$ with $y\in\Ot$ written as
\begin{equation}
	\label{eq:EIM_appr}
	\Inter_q [f_{y}](x) = \sum_{j=1}^q g_j(y) h_j(x),\quad  x\in\Oo,
\end{equation}
is defined by 
\begin{equation}
	\label{eq:EIM_inter}
	\Inter_q [f_{y}](x_i) =  f_{y}(x_i),\quad  i=1,\ldots,q.
\end{equation}
Thus, \eqref{eq:EIM_inter} is equivalent to the following linear system
\begin{equation}
	\label{eq:EIM_interp}
	\sum_{j=1}^q g_j(y) h_j(x_i) = f_y(x_i),\quad i=1,\ldots,q.
\end{equation}
One of the problems is to ensure that the system above is uniquely solvable, i.e.\ that the matrix~$(h_j(x_i))_{i,j}$ is invertible, which will be considered in the design of the interpolation scheme.
\subsection{Algorithm}

The construction of the basis functions and interpolation points is based on a greedy algorithm.
Note that the EIM is defined with respect to a given norm on $\Oo$ and we consider here $L^p(\Oo)$-norms for $1\le p\le \infty$. The algorithm is given as follows.

\mybox{
Empirical Interpolation Method
}
{
		Set $q=1$.		
		Do while $\err < \tol$:
		\begin{itemize}
				\item[\tt a.]	
				Pick the sample point
				\begin{equation}
					\label{eq:EIM_argmax1}
					y_{q} = \argsup{y\in\Ot} \|f_{y} - \Inter_{q-1} [f_{y}] \|_{L^p(\Oo)}, 
				\end{equation}
				and the corresponding interpolation point
				\begin{equation}
					\label{eq:EIM_argmax2}
					 x_{q} = \argsup{x\in\Oo} |f_{y_q}(x) - \Inter_{q-1} [f_{y_q}](x)|.
				\end{equation}
				\item[\tt b.]	
				Define the next basis function as 
				\begin{equation}
					\label{eq:hq}
					h_q = \frac{f_{y_q} - \Inter_{q-1} [f_{y_q}]}{f_{y_q}(x_q) - \Inter_{q-1} [f_{y_q}](x_q)}.
				\end{equation}
				\item[\tt c.]	
				Define the error level by
				\[
					\err =\| \err_p \|_{L^\infty(\Ot)} \quad\mbox{with}\quad  \err_p(y) = \| f_y - \Inter_{q-1} [f_y] \|_{L^p(\Oo)},
				\]
				and set $q:=q+1$.
		\end{itemize}
}
\begin{remark}
Note that whenever $\mbox{dim}(\myspan{\mathcal F})=q^\star$, the algorithm finishes for $q=q^\star$.
\end{remark}

As long as $q\le q^\star$, note that the basis functions $\{h_1,\ldots,h_q\}$ and the snapshots $\{f_{y_1},\ldots,f_{y_q}\}$ span the same space, i.e.,
\[
	\Vq=\myspan{h_1,\ldots,h_q} = \myspan{f_{y_1},\ldots,f_{y_q}}.
\]
The former are preferred to the latter due the following properties
\begin{equation}
	\label{eq:EIM_LT}
	h_i(x_i) = 1,\quad\forall i=1,\ldots,q
	\quad\quad\mbox{and}\quad	h_j(x_i) = 0, \quad 1\le i < j \le q. 
\end{equation}

\begin{remark}
It is easy to show that the interpolation operator $\Inter_q$ is the identity if restricted to the space $\Vq$, i.e.,
\[
	\Inter_q[f_{y_i}](x) = f_{y_i}(x),\quad i=1,\ldots,q,\quad x\in \Oo,\\
\]
\end{remark}
\begin{remark}
The construction of the interpolating functions and the associated interpolation points follows a  greedy approach: we add the function in $\mathcal F$ that is the worse approximated by the current interpolation operator and the interpolation point is where the error is the largest. The construction is thus recursive which, in turn, means that it is of low computational cost.
\end{remark}

\begin{remark}
As explained in \cite{Barrault:2004wb}, the algorithm can be reduced to the selection of the interpolation points only, in the case where the family of interpolating functions $\{f_{y_1},\ldots,f_{y_q},\ldots \}$ is preexisting. This can be the case for instance if a POD strategy has been used previously or when one considers a set that has a canonical basis and ordering (like the set of polynomials).
 \end{remark}

Note that solving the interpolation system \eqref{eq:EIM_interp} can be written as a linear system ${\rm B}\,{\rm g}_y ={\rm f}_y$  with $q$ unknowns and equations where
\[
	{\rm B}_{i,j} = h_j(x_i), \quad 
	({\rm f}_y)_i = f_y(x_i),
	\quad  i,j=1,\ldots,q,
\]
such that the interpolant is defined by
\[
	\Inter_q [f_{y}](x) = \sum_{j=1}^q ({\rm g}_y)_j h_j(x),\quad  x\in\Oo.
\]
This construction of the basis functions and interpolation points satisfies the following theoretical properties (see \cite{Barrault:2004wb}):
\begin{itemize}
\item The basis functions $\{h_1,\ldots,h_q\}$ consist of linearly independent functions;
\item The interpolation matrix ${\rm B}_{i,j}$ is lower triangular with unity diagonal and hence invertible by~\eqref{eq:EIM_LT};
\item The empirical interpolation procedure is well-posed in $L^p(\Oo)$, as long as $q\le q^\star$.
\end{itemize}
If the $L^\infty(\Oo)$-norm ($p=\infty$) is considered, the error
analysis of the interpolation procedure classically involves the
Lebesgue constant $\Lambda_q = \sup_{x\in\Oo} \sum_{i=1}^q |L_i(x)|$
where $L_i\in\Vq  $ are the Lagrange functions satisfying $L_i(x_j)=\delta_{ij}$. The following bound holds \cite{Barrault:2004wb}
\[
	\| f_y - \Inter_q[f_y] \|_{L^\infty(\Oo)} 
	\le (1+\Lambda_q) \inf_{v_q\in\Vq} \| f_y - v_q] \|_{L^\infty(\Oo)}.
\]
An (in practise very pessimistic) upper bound (c.f.\ \cite{Maday:2009tg}) of the Lebesque constant is given by
\[
	\Lambda_q\le 2^q-1,
\]
which in turn results in the following estimate. 
Assume that $\mathcal F\subset\mathcal X\subset L^\infty(\Oo)$ and that there exists a sequence of finite dimensional spaces
\[
	\mathbb Z_1\subset \mathbb Z_2 \subset \ldots,
	\quad \mbox{dim}(\mathbb Z_q) = q,
	\quad \mbox{and} \quad 
	\mathbb Z_q\subset \mathcal F,
\]
such that there exists $c>0$ and $\alpha >\log(4)$ with
\[
	\inf_{v_q\in\mathbb Z_q} \|f_y-v_q \|_{\mathcal X} \le ce^{-\alpha q},
	\quad y\in\Ot,
\]
then
\[
	\| f_y-\Inter_q[f_y]\|_{L^\infty(\Oo)}
	\le ce^{-(\alpha-\log(4)) \, q}.
\]

\begin{remark}
The worst-case situation where the Lebesgue constant scales indeed  like $\Lambda_q\le 2^q-1$ is rather artificial and in all implementations we have done so far involving functions belonging to some reasonable set with small Kolmogorov $n$-width, the growth of the Lebesgue constant is much more reasonable and in most of the times a linear growth is observed. 
Note that, the points that are generated by the EIM using polynomial basis functions (in increasing order of degree) on $[-1,1]$ are exactly the Leja points as indicated in the frame of the EIM by A. Chkifa\footnote{personal  communication} and the discussion in Section~\ref{ssec:constrip} in the case of ACA.
On the other hand, if one considers the Leja points on a unit circle and then project them onto the interval $[-1,1]$ a linear growth is shown in~\cite{chkifa2012lebesgue}.
 \end{remark}

\subsection{Practical implementation}
In the practical implementation of the EIM one encounters the following problem.
Finding the supremum respectively the arg sup in 	\eqref{eq:EIM_argmax1} and \eqref{eq:EIM_argmax2} is not feasible if any kind of approximation is effected. 
The least difficult way, but not the only one, is to consider representative point-sets $\Oot=\{\hat x_1,\hat x_2,\ldots,\hat x_M\}$ of $\Oo$ and $\Ott=\{\hat y_1,\hat y_2,\ldots,\hat y_N\}$ of $\Ot$.
Then, the EIM is written as:
\mybox{
Empirical Interpolation Method (possible implementation of EIM)
}
{
		Set $q=1$.		
		Do while $\err < \tol$:
		\begin{itemize}
				\item[\tt a.]	
				Pick the sample point
				\begin{equation}
					\label{eq:EIM_argmax1_pract}
					y_{q} = \argmax{y\in\Ott} \|f_{y} - \Inter_{q-1} [f_{y}] \|_{L^p(\Oo)}, 
				\end{equation}
				and the corresponding interpolation point
				\[
					 x_{q} = \argmax{x\in\Oot} |f_{y_q}(x) - \Inter_{q-1} [f_{y_q}](x)|.
				\]
				\item[\tt b.]	
				Define the next basis function as 
				\[
				h_q = \frac{f_{y_q} - \Inter_{q-1} [f_{y_q}]}{f_{y_q}(x_q) - \Inter_{q-1} [f_{y_q}](x_q)}.
				\]
				\item[\tt c.]	
				Define the error level by
				\[
					\err =\| \err_p \|_{L^\infty(\Ot)} \quad\mbox{with}\quad  \err_p(y) = \| f_y - \Inter_{q-1} [f_y] \|_{L^p(\Oo)}
				\]
				and set $q:=q+1$.
		\end{itemize}
}

This possible implementation of the EIM is sometimes referred to as the discrete empirical interpolation method (DEIM) \cite{chaturantabut2009discrete}.

\begin{remark}
	Different strategies have been reported in \cite{Maday:2012uk,Haasdonk:2011dj} to successively enrich the training set~$\Ott$.
	The main idea is to start with a small number of training points and enrich the set during the iterations of the algorithm and obtain a very fine discretization only towards the end of the algorithm.
	One can also think of enriching the training set $\Oot$ simultaneously.
\end{remark}

\begin{remark}
Using representative pointsets $\Oot$ and $\Ott$ is only one way to discretize the problem. Alternatively, one can think of using optimization methods to find the maximum over $\Oo$ and $\Ot$. Such a strategy has been reported in \cite{BuiThanh2,BuiThanh1} in the context of the reduced basis method, which, as well as the EIM, is based on a greedy algorithm.
\end{remark}

\subsection{Practical implementation using the matrix representation of the function}
One can define an implementation of the EIM in a completely discrete setting using the representative matrix of $f$ defined by $\rM_{i,j}=f(x_i,y_j)$ for $1\le i\le M$ and $1\le j\le N$. 
For the sake of short notation we recall the notation $\rM_{:,j}$ used for the $j$-th column of $\rM$.

Assume that we are given a set of basis vectors $\{\rh_1,\ldots,\rh_q\}$ and interpolation indices $i_1,\ldots,i_q$, the discrete interpolation operator $\MInter_{q}:\real^N\to\real^N$ of column vectors is given in the span of the basis vectors $\{\rh_j\}_{j=1}^q$, i.e.\ by $\MInter_{q} [\rr] = \sum_{j=1}^q g_j(\rr) \rh_j $ for some scalars $g_j(\rr)$, such that
\[
	(\MInter_{q} [\rr])_{i_k} = \sum_{j=1}^q g_j(\rr)\, (\rh_j)_{i_k} = \rr_{i_k},\quad \rr\in\real^N,\quad k=1,\ldots,q .
\]
Using this notation, we then present the discrete version of the EIM:
\mybox{
Empirical Interpolation Method (implementation based on representative matrix $\rM$ of $f$)
}
{
		Set $q=1$.		
		Do while $\err < \tol$
		\begin{itemize}
				\item[\tt a.]	
				Pick the sample index
				\[
					j_{q} = \argmax{j=1,\ldots,M} \|\rM_{:,j} - \MInter_{q-1} [\rM_{:,j}]\|_{\ell^p},
				\]
				and the corresponding interpolation index
				\[
					 i_{q} = \argmax{i=1,\ldots,N}|\rM_{i,j_q} - (\MInter_{q-1} [\rM_{:,j_q}])_{i} |.
				\]
				\item[\tt b.]	
				Define the next approximation column by
				\[
					\rh_q = \frac{\rM_{:,j_q} - \MInter_{q-1} [\rM_{:,j_q}]}{\rM_{i_q,j_q} - (\MInter_{q-1} [\rM_{:,j_q}])_{i_q}}.
				\]
				\item[\tt c.]	
				Define the error level by
				\[
					\err = \max_{j=1,\ldots,M} \|\rM_{:,j} - \MInter_{q-1} [\rM_{:,j}] \|_{\ell^p}
				\]
				and set $q:=q+1$.
		\end{itemize}
}
This procedure allows to define an approximation of any coefficient of the matrix $\rM$.
In some cases however, one would like to obtain an approximation of $f(x,y)$ for {\it any} $(x,y)\in\Oo\times\Ot$.
After running the implementation, one can still construct the continuous interpolant $\Inter_Q[f](x,y)$ for {any} $(x,y)\in\Oo\times\Ot$.
Indeed, the interpolation points $x_1,\ldots,x_Q$ are provided by $x_q = \hat x_{i_q}$.
The construction of the (continuous) basis functions $h_q$ is based on mimicking part {\tt b} of the discrete algorithm but in a continuous context.
Therefore, during the discrete version one saves the following data 
\begin{align*}
	s_{q,j} &= g_j(\rM_{:,j_q}),&&
	\mbox{from}\quad 
	\MInter_{q-1} [\rM_{:,j_q}]=\sum_{j=1}^{q-1} g_j(\rM_{:,j_q})\, \rh_j,\\
	s_{q,q} &= \rM_{i_q,j_q} - (\MInter_{q-1} [\rM_{:,j_q}])_{i_q}.&&
\end{align*}
Then, the continuous basis functions can be recovered by the following recursive formula
\[
	h_q = \frac{ f_{y_q} - \sum_{j=1}^{q-1} s_{q,j} \, h_j }{s_{q,q}}
\]
using the notation $y_q = \hat y_{i_q}$.

\subsection{Generalizations of the EIM}
In the following, we present some generalizations of the core concept behind the EIM.

\subsubsection{Generalized Empirical Interpolation Method (gEIM)} 
We have seen that the EIM-interpolation operator $\Inter_q[f_y]$, $y\in\Ot$, interpolates the function $f_y$ at some empirically constructed points $x_1,\ldots,x_q$.
The EIM can be generalized in the following sense as proposed in \cite{MagenesGEIM}. Let $\Sigma$ be a dictionary of  linear continuous forms (say for the $L^2(\Oo)$-norm) acting on functions $f_y$, $y\in\Ot$. 
Then, the gEIM consists in providing a set of basis functions $h_1,\ldots,h_q$, such that $\mathbb V_q = \myspan{h_1,\ldots,h_q} = \myspan{f_{y_1},\ldots,f_{y_q}}$ for some empirically chosen $\{y_1,\ldots,y_q\}\subset\Ot$, and a set of linear forms, or moments, $\{\sigma_1,\ldots,\sigma_q\}\subset \Sigma$. 
The generalized interpolant then takes the form
\[
	\gInter_q[f_y] = \sum_{j=1}^q  g_j(y) h_j(x), \quad x\in\Oo, \quad y\in\Ot,
\]
and is defined in the following way
\[
	\sigma_i\left(\gInter_q[f_y]\right) = \sigma_i(f_y),\quad
	 i=1,\ldots,q,
\]
which will define the coefficients $g_j(y)$ for each $y\in\Ot$.
We note that if the linear forms are Dirac functionals $\delta_x$ with $x\in\Oo$, then the gEIM reduces to the plain EIM.
The algorithm is given by
\mybox{
Generalized Empirical Interpolation Method (gEIM)
}
{
		Set $q=1$.		
		Do while $\err < \tol$:
		\begin{itemize}
				\item[\tt a.]	
				Pick the sample point
				\[
					y_{q} = \argsup{y\in\Ot} \|f_{y} - \gInter_{q-1} [f_{y}] \|_{L^p(\Oo)}, 
				\]
				and the corresponding interpolation moment
				\[
					 \sigma_{q} = \argsup{\sigma\in\Sigma} |\sigma(f_{y_q} - \gInter_{q-1} [f_{y_q}])|.
				\]
				\item[\tt b.]	
				Define the next basis function as 
				\[
					h_q = \frac{f_{y_q} - \gInter_{q-1} [f_{y_q}]}{\sigma_q(f_{y_q} - \gInter_{q-1} [f_{y_q}])}.
				\]
				\item[\tt c.]	
				Define the error level by
				\[
					\err =\| \err_p \|_{L^\infty(\Ot)} \quad\mbox{with}\quad  \err_p(y) = \| f_y - \Inter_{q-1} [f_y] \|_{L^p(\Oo)}
				\]
				and set $q:=q+1$.
		\end{itemize}
}

This constructive algorithm satisfies the following theoretical properties (see \cite{MagenesGEIM}):
\begin{itemize}
\item The set $\{h_1,\ldots,h_q\}$ consists of linearly independent functions;
\item The generalized interpolation matrix $({\rm B})_{ij}=\sigma_i(h_j)$ is lower triangular with unity diagonal (hence invertible) with other entries $s\in [-1, 1]$;
\item The generalized empirical interpolation procedure is well-posed in $L^2(\Oo)$.
\end{itemize}
In order to quantify the error of the interpolation procedure, like in the standard interpolation procedure, we introduce the Lebesgue constant in the $L^2$-norm:	
\[
	\Lambda_q	=	\sup_{y\in\Ot} \frac{\|\gInter_q[f_y]\|_{L^2(\Oo)}}{\|f_y\|_{L^2(\Oo)}},
\]
i.e.\ the $L^2$-operator norm of $\gInter_q$. Thus, the interpolation error satisfies:
\[
	\|f_y - \gInter_q [f_y]\|_{L^2(\Omega)} 
	\le (1+\Lambda_q)	\inf_{v_q\in \mathbb V_q} \|f_y - v_q\|_{L^2(\Omega)}.
\]
Again, a (very pessimistic) upper-bound for $\Lambda_q$ is: 
\[
	\Lambda_q \le 2^{q-1} \max_{i = 1,\ldots,q} \|h_i\|_{L^2(\Omega)},\quad 
\]
indeed, the Lebesgue constant is, in many cases, uniformly bounded in this generalized case. The following result proves that the greedy construction is quite optimal
\cite{madaypriori}

\begin{enumerate}
\item[\tt 1.] Assume that the Kolmogorov $n$-width of $\mathcal F$ in $L^2(\Oo)$ is upper bounded by $C_0 n^{-\alpha}$ for any $n\geq 1$, then the interpolation error of the gEIM greedy selection process 
satisfies for any $f \in \mathcal F$ the inequality $\Vert f(\cdot,y) -J_Q[f(\cdot,y)] \Vert _{L^2(\Oo)} \leq C_0 (1+\Lambda_Q)^3 Q^{-\alpha}$.
\item[\tt 2.] Assume that the Kolmogorov $n$-width of $\mathcal F$ in $L^2(\Oo)$ is upper bounded by $C_0 e^{-c_1 n^{\alpha}}$ for any $n\geq 1$, then the interpolation error of the gEIM greedy selection process 
satisfies for any $f \in \mathcal F$ the inequality $\Vert f(\cdot,y) -J_Q[f(\cdot,y)] \Vert _{L^2(\Oo)} \leq  C_0 (1+\Lambda_Q)^3  e^{-c_2 Q^{\alpha}}$ for a positive constant $c_2$ slightly smaller than $c_1$.
\end{enumerate}

\subsubsection{$hp$-EIM}
If the Kolmogov $n$-width is only decaying slowly with respect to $n$ and the resulting number of basis functions and associated integration points is larger than desired, a remedy consists of partitioning the space $\Ot$ into different elements $\Otp{1},\ldots,\Otp{P}$ on which a separate interpolation operator $\Inter_{q_p}:\{f_y\}_{y\in \Oop{p}}\to \mathbb V_{q_p}$ with $p=1,\ldots,P$ is constructed. That is, for each element $\Otp{p}$ a standard EIM as described above is performed. 
The choice of creating the partition is subject to some freedom and different approaches have been presented in \cite{Eftang:2011vn,Fares:2011p6716}.

A somewhat different approach is presented in \cite{Maday:2012uk}, although in the framework of a projection method,  where the idea of a strict partition of the space $\Ot$ is abandoned. Instead, given a set of sample points $y_1,\ldots,y_K$ for which the basis functions $f(\cdot,y_1),\ldots,f(\cdot,y_K)$ are known (or have been computed) 
a local approximation space for any $y\in\Ot$ is constructed by considering the $N$ basis functions whose parameter values are closest to $y$. 
In addition, the distance function, measuring the distance between two points in $\Ot$, can be empirically built in order to represent local anisotropies in the parameter space $\Ot$.
Further, the distance function can also be used to define the training set $\Ott$ which can be uniformly sampled with respect to the problem dependent distance function.

\subsubsection{Curse of high-dimensionality}
Several approaches have been presented in cases where $\Ot$ is high-dimensional ($\mbox{dim}(\Ot)\approx 10$).
In such cases, finding the maximizer in \eqref{eq:EIM_argmax1_pract} becomes a challenge. Since the discrete set $\Ott$ should be representative of $\Ot$, we require that $\Ott$ consists of a very large number of training points. Finding the maximum over this huge set is therefore prohibitive expensive as a result of the curse of dimensionality.

In \cite{brown_sc_2011_15}, the authors propose a computational approach that randomly samples the space $\Ot$ with a feasible number of training points, that is however changing over the iterations. Therefore, counting all tested training points over all iterations is still a very large number, at each iteration though finding the maximum is a feasible task.

In \cite{brown_sc_2011_31}, the authors use, in the framework of the reduced basis method, an ANOVA expansion based on sparse grid quadrature in order to identify the sensitivity of each dimension in $\Ot$. Then, once unimportant dimensions in $\Ot$ are identified, the values of the unimportant dimensions are fixed to some reference value and the variation of $y$ in $\Ot$ is then restricted to the important dimensions. Finally, a greedy-based algorithm is used to construct a low-order approximation.

\section{Comparison of ACA vs. EIM}

In the previous sections, we have given independent presentations of the basics of the ACA and the EIM type methods. As was explained, the backgrounds and the applications are different. In addition, we have also presented the results of the convergence analysis of these approximations yielding another fundamental difference between the two approaches. The frame for the convergence of the ACA is a comparison to any other interpolating system, such as the polynomial approximation and the existence of derivatives for the family of functions $f_y$, $y\in\Ot$ is then the reason for convergence.
The convergence of the EIM is compared with respect to the $n$-width expressed by the Kolmogorov small dimension.

Nevertheless, despite there differences in origins, it is clear that some link exist between these two constructive approximation methods. We show now the relation between the ACA and the EIM in a particular case.
\begin{thm}
The Bivariate Adaptive Cross Approximation with global pivoting is equivalent to the Empirical Interpolation Method using the $L^\infty(\Oo)$-norm.
\end{thm}
\begin{proof}
We proceed by induction. Our affirmation ${\tt A}_q$ at the $q$-th step is: 
\begin{itemize}
	\item[$({\tt A}_q)_{\rm 1}$:] The interpolation points $\{x_1,\ldots,x_q\}$ and $\{y_1,\ldots,y_q\}$ of the EIM and ACA are identical;
	\item[$({\tt A}_q)_{\rm 2}$:] $g_q(y) = r_{q-1}(x_q,y), \quad y\in\Ot,$;
	\item[$({\tt A}_q)_{\rm 3}$:] $\Inter_{q}[f_y](x) = \mathfrak{I}_q[f_y](x), \quad  (x,y)\in \Oo\times\Ot.$
\end{itemize}
\vspace*{6pt}
\noindent
{\tt 1:Induction base ($q=1$):} 
First, we note that $r_0=f$ and thus
\[
	(x_1,y_1)
	=\argsup{(x,y)\in\Oo\times\Ot}\,|r_0(x,y)|
	=\argsup{(x,y)\in\Oo\times\Ot}\,|f(x,y)|.
\]
Then, from \eqref{eq:hq} we conclude that $h_1(x) = \frac{f(x,y_1)}{f(x_1,y_1)}$ and by \eqref{eq:EIM_interp} we obtain that 
$g_1(y)=\frac{f(x_1,y)}{h_1(x_1)}=f(x_1,y)=r_0(x_1,y)$ since $h_1(x_1) = 1$.
Further, using additionally \eqref{eq:EIM_appr}, we get
\[
	\Inter_{1}[f_{y_1}](x) 
	= g_1(y)h_1(x)
	= r_0(x_1,y)\frac{f(x,y_1)}{f(x_1,y_1)}
	= \frac{r_0(x_1,y)r_0(x,y_1)}{r_0(x_1,y_1)}
	= \mathfrak{I}_1[f_y](x),
	\quad (x,y)\in\Oo\times\Ot,
\]
and ${\tt A}_1$ holds in consequence.

\vspace*{6pt}
\noindent
{\tt 2:Induction step ($q>1$):} Let us assume ${\tt A}_{q-1}$ to be true and we first note that 
\begin{equation}
	\label{eq:R_eq_Int}
	r_{q-1}(x,y) 	=  f(x,y)-\Inter_{q-1}[f_y](x)
\end{equation}
by \eqref{eq:rel_f_Int} and $({\tt A}_{q-1})_{\rm 3}$. Therefore, the selection criteria for the points $(x_q,y_q)$ are identical for the EIM wit $p=\infty$ and  the ACA with global pivoting. 
In consequence, the chosen sample points $(x_q,y_q)$ are identical.
Further, combining \eqref{eq:hq} and \eqref{eq:R_eq_Int} yields
\begin{equation}
	\label{eq:hq_eq_Rq}
	h_q(x) = \frac{f_{y_q}(x) - \Inter_{q-1} [f_{y_q}](x)}{f_{y_q}(x_q) - \Inter_{q-1} [f_{y_q}](x_q)}
	= \frac{r_{q-1}(x,y_q)}{r_{q-1}(x_q,y_q)}.
\end{equation}
By \eqref{eq:EIM_interp} for $i=q$, using that $h_q(x_q)=1$ and \eqref{eq:R_eq_Int}, we obtain $({\tt A}_q)_{\rm 2}$:
\begin{equation}
	\label{eq:gq_eq_Rq}
	g_q(y) 
	=	f(x_q,y) - \sum_{j=1}^{q-1} g_j(y) h_j(x_q)
	= f(x_q,y) - \Inter_{q-1}[f_y](x_q) 
	= r_{q-1}(x_q,y). 
\end{equation}
Finally, combining \eqref{eq:hq_eq_Rq} and \eqref{eq:gq_eq_Rq} in addition to $({\tt A}_{q-1})_{\rm 3}$, we conclude that
\[
	\Inter_{q}[f_y](x) 
	= \Inter_{q-1}[f_y](x)   +  g_q(y) h_q(x)
	= \mathfrak{I}_{q-1}[f_y](x)  + r_{q-1}(x_q,y)\frac{r_{q-1}(x,y_q)}{r_{q-1}(x_q,y_q)}
	= \mathfrak{I}_{q}[f_y](x) 
\]
and the proof is complete.
\end{proof}

\section{Gappy POD}
In the following, we present a completion to the POD method called {\it gappy POD}  \cite{everson1995karhunen, Willcox:2006gd,bui2004aerodynamic} or {\it missing point estimation} \cite{Astrid:da}.
We refer to it as the gappy POD in the following.
It is a projection based method (thus not an interpolation based method although in some particular cases it can be interpreted as an interpolation scheme).
However, the projection matrix is approximated by a low-rank approximation that in turn is based on partial or incomplete (``gappy'') data of the functions under consideration. 
In a first turn, we present the method as introduced in \cite{Willcox:2006gd,bui2004aerodynamic} and we generalize it in a second turn.

\subsection{The gappy POD algorithm}
We start from the conceptual idea that a set of basis functions $\{h_1,\ldots,h_Q\}$, that can --- but does not need to --- be obtained through a POD procedure, is given.
We first introduce the idea of gappy POD in the context of Remark \ref{rem:SVD} where functions are represented by a vector containing its pointwise values on a given grid $\Oot=\{\hat x_1,\ldots,\hat x_M\}$. We remind that the projection $\proj_Q[f_y]$ of $f_y$ with $y\in\Ot$ onto the space spanned by $\{h_1,\ldots,h_Q\}$ is defined by 
\[
	(\proj_Q[f_y],h_q)_{\Oot} = (f_y,h_q)_{\Oot},\quad  q=1,\ldots,Q.
\]
Next, assume that we only dispose of some incomplete data of $f_y$. That is, we 
are given say $L$ ($< M$) distinct points $\{x_1,\ldots,x_L\}$ among $\Oot$ where $f_y(x_i)$ is available.
Then, we define the gappy scalar product by 
\[
			(v,w)_{L,\Oot} = \frac{|\Oo|}{L}\sum_{i=1}^{L} v(x_i)\,w(x_i),
\]
which only takes into account available data of $f_y$. 
We can compute the gappy projection defined by
\[
	(\proj_{Q,L}[f_y],h_q)_{L,\Oot} = (f_y,h_q)_{L,\Oot},\quad  q=1,\ldots,Q.
\]
Observe that the basis functions $\{h_1,\ldots,h_Q\}$ are no longer orthonormal for the gappy scalar product and 
that the stability of the method mainly depends on the properties of the mass matrix $\rM_{h,L}$ defined by
\[
	(\rM_{h,L})_{i,j} = (h_j,h_i)_{L,\Oot}.
\]

To summarize, in the above presentation we assumed that the data of $f_y$ at some given points was available and then defined a ``best approximation'' with respect to the available but incomplete data. For instance, the data can be assimilated by physical experiments and the gappy POD allows to reconstruct the solution in the whole domain $\Oot$ assuming that it can be accurately represented by the basis functions $\{h_1,\ldots,h_Q\}$.

We now change the viewpoint and ask the question: If we can place $L$ sensors at the locations $\{ x_i\}_{i=1}^L\subset \Oo$ at which we have access to the data $f_y( x_i)$ (through measurements), where would we place the points $\{ x_i\}_{i=1}^L$?

One might consider different criteria to chose the sensors. In \cite{Willcox:2006gd} the placement of $L$ sensors is stated as a minimization problem
\begin{align*}
	\min \kappa(\rM_{h,L}) \quad 
	 \mbox{ where }\quad
	  \rM_{h,L}\mbox{ is based on $L$ points }\{x_1,\ldots,x_L\}
\end{align*}
and $\kappa(\rM_{h,L})$ denotes the condition number of $\rM_{h,L}$.
We now report a slight modification of the algorithm presented in \cite{Willcox:2006gd,Astrid:da} to construct a sequence of sensor placements $\{x_1,\ldots,x_L\}$ (with $L\ge Q$) based on an incremental greedy algorithm.

\mybox{
Sensor placement algorithm with gappy POD and minimal condition number
}
{
For $1\le l\le L$: 
\[
	x_{l} = \argmin{x\in\Oo} \kappa(\rM_{h,l}(x))
\]
where 
\[
	(\rM_{h,l}(x))_{i,j} = \frac{|\Oo|}{l}\left[ 
	\sum_{k=1}^{l-1} f_i(x_k)f_j(x_k) + f_i(x) f_j(x)
	\right],\quad 1\le i,j \le \min(Q,l).
\]
}
This natural algorithm actually seems to have some difficulties at the beginning, for small values of $l$. It is thus recommended to start with this other algorithm

\mybox{
Sensor placement algorithm with gappy POD and minimal error
}
{
For $1\le l\le L$: 
\[
	x_{l} = \argmax{x\in\Oo} \| \proj_{Q,l-1} [f_y] - f_y  \|_{L^p(\Oo)}
\]
where $\proj_{Q,l-1} [f_y] $ is the gappy projection of $f_y$ onto the span of \{$h_1,\ldots,h_{l-1}\}$ based on the pointwise information at $\{x_1,\ldots,x_{l-1}\}$.
}

This criterion is actually the one that is used in the gappy POD method presented in \cite{carlberg2013gnat} in the frame of the GNAT approach that allows a stabilized implementation of the gappy method for a challenging CFD problem.
Further, we have the following link between the gappy projection and the EIM as noticed in \cite{galbally2010non}.

\begin{lem}
Let $\{h_1,\ldots,h_Q\}$ and  $\{x_1,\ldots,x_Q\}$ be given basis functions and interpolation nodes.
If the interpolation is well-defined (i.e.\ the interpolation matrix being invertible), then the interpolatory system based on the basis functions $\{h_1,\ldots,h_Q\}$ and the interpolation nodes $\{x_1,\ldots,x_Q\}$ is equivalent to the 
gappy projection system based on the basis functions $\{h_1,\ldots,h_Q\}$ with available data at the points $\{x_1,\ldots,x_Q\}$, that is, for any $y\in\Ot$ the unique interpolant $\Inter_Q [f_y]\in \myspan{h_1,\ldots,h_Q}$  such that
\begin{equation}
	\label{eq:IntGappy}
	\Inter_Q [f_y](x_q) = f_y(x_q), \quad q=1,\ldots,Q,
\end{equation}
is equivalent to the unique gappy projection $\proj_{Q,L}[f_y]$ defined by
\begin{equation}
	\label{eq:ProjGappy}
	(\proj_{Q,L}[f_y],h_q)_{Q,\Oo} = (f_y,h_q)_{Q,\Oo}, \quad q=1,\ldots,Q.
\end{equation}
\end{lem}
\begin{proof}
	Multiply \eqref{eq:IntGappy} by $\frac{|\Oo|}{Q} h_i(x_q)$ and take the sum over all $q=1,\ldots,Q$ to obtain
	\[
		\frac{|\Oo|}{Q} \sum_{q=1}^Q\Inter_Q [f_y](x_q) h_i(x_q)= \frac{|\Oo|}{Q} \sum_{q=1}^Qf_y(x_q) h_i(x_q), \quad i=1,\ldots,Q,
	\]
	which is equivalent to $(\Inter_Q [f_y],h_i)_{L,\Oot} = (f_y,h_i)_{L,\Oot}$ for all $i=1,\ldots,Q$.
	On the other hand, if $\proj_{Q,L}[f_y]$ is the solution of \eqref{eq:ProjGappy}, then there holds that 
	\begin{equation}
		\label{eq:ProjIntGappy}
		\sum_{q=1}^Q\proj_{Q,L} [f_y](x_q) h_i(x_q)=  \sum_{q=1}^Qf_y(x_q) h_i(x_q), \quad i=1,\ldots,Q.
	\end{equation}
	Since the interpolating system is well-posed, the interpolation matrix ${\rm B}_{i,j} = h_j(x_i)$ is invertible and thus there exists a vector ${\rm u}_j$ such that ${\rm B}\, {\rm u}_j = \re_j$ for some $j=1,\ldots,Q$ where $\re_j$ is the canonical basis vector. Then, multiply \eqref{eq:ProjIntGappy} by $({\rm u}_j)_i$ and sum over all $i$:
\[
			\sum_{i,q=1}^Q\proj_{Q,L} [f_y](x_q)\, ({\rm u}_j)_i \,{\rm B}_{qi}=  \sum_{i,q=1}^Qf_y(x_q) \,({\rm u}_j)_i \,{\rm B}_{q,i}, \quad j=1,\ldots,Q,
\]
to get
\[
			\proj_{Q,L} [f_y](x_j)=  f_y(x_j) , \quad j=1,\ldots,Q.
\]
Thus, the gappy projection satisfies the interpolation scheme.
\end{proof}

One feature of the sensor placement algorithm based on the gappy POD framework is that the basis functions $\{h_1,\ldots,h_q\}$ are given and the sensors are chosen accordingly.
As a consequence of the interpretation of the gappy projection as an interpolation scheme if the number of basis functions and sensors coincide, one might combine the gappy POD approach with the EIM in the following way in order to construct basis functions and redundant sensor locations simultaneously:
\begin{itemize}
\item[\tt 1.] 
Use the EIM to construct simultaneously $Q$ basis functions $\{h_q\}_{q=1}^Q$ and interpolation points $\{x_q\}_{q=1}^Q$ until a sufficiently small error is achieved;
\item[\tt 2.]
Use the gappy projection framework as outlined above to add interpolation points (sensors) to enhance the stability of the scheme. 
\end{itemize}

\subsection{Generalization of gappy POD}
In the previous algorithm the functions were represented by their nodal values at some points $\hat x_1,\ldots,\hat x_M$.
That is, we can introduce for each point $\hat x_i$ a functional $\hat\sigma_i=\delta_{\hat x_i}$ ($\delta_x$ denoting the Dirac functional associated to the point $x$) such that the interpolant of any continuous function $f$ onto the space $\VN$ of piecewise linear and globally continuous functions can be written as
\[
	\sum_{m=1}^M \hat\sigma_m(f) \, \hat\varphi_m,
\]
where $\{\hat\varphi_m\}_{m=1}^M$ denotes the Lagrange basis of $\VN$ with respect to the points $\hat x_1,\ldots,\hat x_M$.

We present now a generalization where we allow a more general discrete space $\VN$. 
Therefore, let
$\VN$ be a $\cN$-dimensional discrete space spanned by a set of basis functions $\{\hat\varphi_i\}_{i=1}^\cN$ such as for example the finite element hat-functions, Fourier-basis or polynomial basis functions.
In the context of the theory of finite elements, c.f.~\cite{ciarlet1978finite}, we are given $M$ functionals $\{\hat\sigma_m\}_{m=1}^M$, associated with the basis set $\{\hat\varphi_i\}_{i=1}^\cN$, which determine the degrees of freedom of a function.
That is, for $f$ regular enough such that all degrees of freedom $\hat\sigma_m(f) $ are well-defined, the following interpolation scheme 
\[
	f \to \sum_{m=1}^M \hat\sigma_m(f) \hat\varphi_m
\] 
defines a function in $\VN$ that interpolates the degrees of freedom.

We start with noting that the scalar product between two functions $f,g$ in $\VN$ is given by
\[
	(f,g)_\Oo = \sum_{n,m=1}^\cN \hat\sigma_n(f) \, \hat\sigma_m(g) \,(\hat\varphi_n,\hat\varphi_m)_\Oo.
\]

In this framework, the meaning of ``gappy'' data is generalized. We speak of gappy data if only partial data of degrees of freedom, i.e.\ the $\hat\sigma_m(f)$ is available. Thus, in this generalized context, the degrees of freedom are not necessarily nodal values, i.e.\ the functionals being Dirac functionals, and depend on the choice of the basis functions.

Assume that we are given $Q$ basis functions $h_1,\ldots,h_Q$ that describe a subspace in $\VN$ and $L\ge Q$ degrees of freedom $\sigma_l = \hat \sigma_{i_l}$, for $l=1,\ldots,L$ (chosen among all $M$ degrees of freedom $\hat\sigma_1,\ldots,\hat\sigma_L$). 
Denoting by $\varphi_l=\hat\varphi_{i_l}$ the corresponding $L$ basis functions, we then define a gappy scalar product
\[
	(f,g)_{L,\Oo} = \frac{M}{L}\sum_{l,k=1}^L \sigma_l(f) \, \sigma_k(g) \,(\varphi_{l},\varphi_{k})_\Oo.
\] 
Given any $f_y$, $y\in\Ot$, the gappy projection $\proj_{Q,L}[f_y]\in \myspan{h_1,\ldots,h_Q}$ is defined by
\[
	(\proj_{Q,L}[f_y],h_q)_{L,\Oo}  = (f_y,h_q)_{L,\Oo}, \quad  q=1,\ldots,Q.
\] 
Then, the sensor placement algorithm introduced in the previous section can easily be generalized to this setting.

\begin{remark}
If the mass matrix $(\hat\rM)_{i,j} = (\hat\varphi_j,\hat\varphi_i)_\Oo$ associated with the basis set $\{\hat\varphi_i\}_{i=1}^\cN$  satisfies the following orthogonality property $(\hat\rM)_{i,j} = (\hat\varphi_j,\hat\varphi_i)_\Oo = \frac{|\Oo|}{\cN} \delta_{ij}$ either by construction of the basis functions or using mass lumping (in the case of finite elements) and if the basis functions $\{\hat\varphi_i\}_{i=1}^\cN$ are nodal basis functions associated with the set of points $\Oot=\{\hat x_1,\ldots,\hat x_M\}$, then the original gappy POD method is established.
\end{remark}

\begin{remark}
If the mass matrix $(\hat\rM)_{i,j} = (\hat\varphi_j,\hat\varphi_i)_\Oo$ associated with the selected functions $\{\varphi_i\}_{i=1}^L$ is orthonormal, then
the gappy projection $\proj_{Q,L}[f_y]$ is solution to the following quadratic minimization problem
\[
	\min_{f\in \VQ} \sum_{l=1}^L |\sigma_l(f_y)-\sigma_l(f)|^2.
\]
Since $L>Q$ in a general setting, this means that the gappy projection fits the selected degrees of freedom optimally in a least-squares sense.
In the general case, $\proj_{Q,L}[f_y]$ is solution to the following minimization problem
\[
	\min_{f\in \VQ} \sum_{l,k=1}^L (\sigma_l(f_y)-\sigma_l(f)) (\varphi_l,\varphi_m)_{\Oo} (\sigma_k(f_y)-\sigma_k(f)).
\]
\end{remark}


\bibliographystyle{abbrv}
\bibliography{Comp,Comp2}

\end{document}